\documentclass[12pt]{amsart}

\usepackage[utf8]{inputenc}
\usepackage[T1]{fontenc}
\usepackage{lmodern}
\usepackage[margin=1in]{geometry}
\usepackage{amsmath,amssymb,amsthm,amsfonts,mathrsfs}
\usepackage{hyperref}
\usepackage{xcolor}
\usepackage{enumitem}
\usepackage{mathtools}

\hypersetup{
  colorlinks=true,
  linkcolor=blue,
  citecolor=blue,
  urlcolor=blue
}

\newtheorem{theorem}{Theorem}[section]
\newtheorem{proposition}[theorem]{Proposition}
\newtheorem{lemma}[theorem]{Lemma}
\newtheorem{corollary}[theorem]{Corollary}
\newtheorem{definition}[theorem]{Definition}
\newtheorem{remark}[theorem]{Remark}
\newtheorem{example}[theorem]{Example}

\newcommand{\GM}{\mathbb{G}_{\mathrm{m}}}

\newcommand{\C}{\mathbb C}
\newcommand{\Z}{\mathbb Z}
\newcommand{\Der}{\operatorname{Der}}
\newcommand{\Aut}[1]{\operatorname{Aut}}
\newcommand{\End}{\operatorname{End}}
\newcommand{\Spec}{\operatorname{Spec}}
\newcommand{\Frac}{\operatorname{Frac}}
\newcommand{\Span}{\operatorname{Span}}
\newcommand{\SSD}{\operatorname{SSD}}
\newcommand{\LFD}{\operatorname{LFD}}
\newcommand{\id}{\mathrm{id}}

\title{Semisimple derivations, rational slices, and kernels over affine domains}

\author{Luis Cid}
\address{Instituto de Matem\'atica y F\'isica, Universidad de Talca, Casilla 721, Talca, Chile}
\email{luis.cid@utalca.cl}

\date{\today}

\subjclass[2020]{13N15; 14R20, 13A50}
\keywords{semisimple derivation, rational slice, $\GM$-action, invariant ring}

\begin{document}

\begin{abstract}
Let $k$ be an algebraically closed field of characteristic zero and let $B$ be a finitely generated $k$-domain. We study semisimple derivations on $B$, with special emphasis on those whose eigenvalues are integers. For such derivations, after passing to the field of fractions and choosing a rational slice $s$ with $D(s)=s$, we describe the kernel of $D$ explicitly in terms of semi-invariant generators. We also obtain descriptions of the kernel on suitable localizations of $B$ and on $B$ itself by intersection. Several basic properties of semisimple derivations and their behavior under conjugation are also discussed.
\end{abstract}

\maketitle

\section{Introduction}

Let $k$ be an algebraically closed field of characteristic zero, let $B$ be a finitely generated $k$-domain, and set $X=\Spec(B)$, $K=\Frac(B)$.
Semisimple derivations on $B$ are closely related to regular $\GM$-actions on $X$. More precisely, such actions correspond to semisimple derivations whose eigenvalues are integers; see \cite{CiLi22}. If
\[
B=\bigoplus_{m\in \Z} B_m
\]
is the grading induced by a regular $\GM$-action, then the associated derivation $D$ is given by
\[
D(b)=m b \qquad \text{for every } b\in B_m,
\]
and the invariant ring is
\[
B^{\GM}=B_0=\ker(D).
\]

The study of $\GM$-actions is closely related to the linearization problem. In dimension two, regular $\GM$-actions on the affine plane are linearizable \cite{Gu62}, and semisimple automorphisms and semisimple derivations on $k[x,y]$ are linearizable up to conjugacy \cite{FurMa2010,Now94}. In dimension three, algebraic $\C^*$-actions on $\C^3$ are also known to be linearizable; this was established in work of Kaliman, Koras, Makar-Limanov, and Russell \cite{KKMR97}. We mention this background only as motivation, since the present article is concerned with a different aspect of semisimple derivations, namely the structure of their kernels.

For locally nilpotent derivations, slices are a classical tool in the study of kernels and associated additive group actions. Our goal is to show that an analogous multiplicative mechanism is available for semisimple derivations with integral eigenvalues, after passing to the field of fractions and choosing an element
\[
s\in K \qquad \text{such that} \qquad D(s)=s.
\]

If $b_1,\dots,b_n$ are semi-invariant generators satisfying
\[
D(b_i)=\lambda_i b_i,
\]
then the elements
\[
b_i s^{-\lambda_i}
\]
belong to the kernel of $D$, and in suitable localizations they generate the full kernel. This leads to explicit descriptions of $\ker(D)$ in the field of fractions, in localizations of $B$, and in $B$ itself by intersection.

The article is organized as follows. In Section~2 we recall basic facts on locally finite and semisimple derivations and their extensions to localizations and fields of fractions. In Section~3 we discuss semi-invariants and the weight decomposition associated with a semisimple derivation. In Section~4 we describe kernels after adjoining a rational slice in the field of fractions, and we also recover the kernel on suitable localizations and on $B$ itself by intersection.

\section{Preliminaries on locally finite and semisimple derivations}

Throughout the paper, $k$ denotes an algebraically closed field of characteristic zero, and $B$ denotes a finitely generated $k$-domain. We write $K=\Frac(B)$ for its field of fractions.

\begin{definition}
Let $V$ be a $k$-vector space and let $T\in \End_k(V)$.
\begin{enumerate}[label=\textup{(\arabic*)}]
    \item The operator $T$ is called \emph{locally finite} if for every $v\in V$, the vector subspace
    \[
    \Span_k\{T^n(v)\mid n\geq 0\}
    \]
    is finite-dimensional.

    \item The operator $T$ is called \emph{semisimple} if $V$ is the direct sum of its eigenspaces:
    \[
    V=\bigoplus_{\lambda\in k} V_\lambda(T),
    \qquad
    V_\lambda(T):=\{v\in V\mid T(v)=\lambda v\}.
    \]
\end{enumerate}
\end{definition}

\begin{definition}
A derivation $D\in \Der_k(B)$ is called \emph{locally finite} if it is locally finite as a $k$-linear endomorphism of $B$. It is called \emph{semisimple} if it is semisimple as a $k$-linear endomorphism of $B$.
\end{definition}

We denote by $\LFD(B)$ the set of locally finite derivations of $B$, by $\SSD(B)$ the set of semisimple derivations of $B$, and by $\SSD^{*}(B)$ the set of semisimple derivations on $B$ whose eigenvalues are integers.

\begin{remark}
Every semisimple derivation is locally finite. Indeed, since $D$ is semisimple, every $b\in B$ is a finite sum of eigenvectors of $D$, so the iterated images $D^n(b)$ are all contained in the finite-dimensional span of those eigenvectors.
\end{remark}

\begin{lemma}
Let $D\in \Der_k(B)$ and let $S\subseteq B$ be a multiplicative subset such that $0\notin S$. Then $D$ extends uniquely to a derivation
\[
D\colon S^{-1}B\to S^{-1}B
\]
given by
\[
D\!\left(\frac{b}{s}\right)=\frac{D(b)s-bD(s)}{s^2}.
\]
In particular, $D$ extends uniquely to a derivation of $K=\Frac(B)$. By abuse of notation, we denote all these extensions again by $D$.
\end{lemma}

\begin{proof}
This is standard and follows from the quotient rule; see, for instance, \cite[Chapter~1]{Fre06}.
\end{proof}

\begin{lemma}
Let $D\in \Der_k(B)$ and let $0\neq b\in B$ satisfy
\[
D(b)=\lambda b
\]
for some $\lambda\in k$. Then in $K$ one has
\[
D(b^{-1})=-\lambda b^{-1}.
\]
More generally,
\[
D(b^m)=m\lambda b^m \qquad \text{for every } m\in \Z.
\]
\end{lemma}

\begin{proof}
Since
\[
0=D(1)=D(bb^{-1})=D(b)b^{-1}+bD(b^{-1}),
\]
we get
\[
0=\lambda+bD(b^{-1}),
\]
hence
\[
D(b^{-1})=-\lambda b^{-1}.
\]

For $m\ge 0$, the formula follows by induction. For $m=-n<0$ with $n\ge 1$, apply $D$ to
\[
b^n b^{-n}=1.
\]
Using Leibniz' rule and the already known formula for $b^n$, we obtain
\[
0=D(b^n b^{-n})=D(b^n)b^{-n}+b^nD(b^{-n})
=n\lambda\, b^n b^{-n}+b^nD(b^{-n}),
\]
so
\[
0=n\lambda+b^nD(b^{-n}),
\]
and therefore
\[
D(b^{-n})=-n\lambda\, b^{-n}.
\]
Thus $D(b^m)=m\lambda b^m$ for every $m\in\Z$.
\end{proof}

\begin{proposition}
Let $\varphi\in \Aut_k(B)$. If $D\in \SSD(B)$, then
\[
\varphi D\varphi^{-1}\in \SSD(B).
\]
\end{proposition}

\begin{proof}
Let $\{a_i\}_{i\in I}$ be a basis of eigenvectors of $D$, with
\[
D(a_i)=\lambda_i a_i.
\]
Then $\{\varphi(a_i)\}_{i\in I}$ is again a basis of $B$, and
\[
(\varphi D\varphi^{-1})(\varphi(a_i))
=
\varphi(D(a_i))
=
\lambda_i\varphi(a_i).
\]
Hence $\varphi D\varphi^{-1}$ is semisimple.
\end{proof}

\begin{proposition}
Let $D_1,D_2\in \SSD(B)$ and suppose that
\[
D_1D_2=D_2D_1.
\]
Then
\[
D_1+D_2\in \SSD(B).
\]
\end{proposition}

\begin{proof}
Since $D_1$ is semisimple,
\[
B=\bigoplus_{\lambda\in k} B_\lambda,
\qquad
B_\lambda:=\{b\in B\mid D_1(b)=\lambda b\}.
\]
Because $D_1$ and $D_2$ commute, each $B_\lambda$ is stable under $D_2$. Since $D_2$ is semisimple, its restriction to each $B_\lambda$ is semisimple. On the other hand,
\[
D_1|_{B_\lambda}=\lambda\,\id_{B_\lambda},
\]
so
\[
(D_1+D_2)|_{B_\lambda}
=
\lambda\,\id_{B_\lambda}+D_2|_{B_\lambda}
\]
is semisimple. Therefore $D_1+D_2$ is semisimple on $B$.
\end{proof}

\begin{proposition}\label{prop:aD-semisimple}
Let $0\neq D\in \SSD(B)$ and let $a\in \ker(D)$. Then
\[
aD\in \SSD(B)
\quad\Longleftrightarrow\quad
a\in k.
\]
\end{proposition}

\begin{proof}
Assume first that $a\in k$. Since $D$ is semisimple, there exists a basis
$\{b_i\}_{i\in I}$ of $B$ consisting of eigenvectors of $D$, say
\[
D(b_i)=\lambda_i b_i \qquad (i\in I).
\]
Then
\[
(aD)(b_i)=a\lambda_i b_i,
\]
so the same basis diagonalizes $aD$. Hence $aD$ is semisimple.

Conversely, assume that $aD$ is semisimple. Since $0\neq D\in \SSD(B)$,
there exists a nonzero eigenvalue $\lambda\in k$ and a nonzero element
\[
b\in B_\lambda:=\{x\in B\mid D(x)=\lambda x\}.
\]
Because $a\in \ker(D)$, for every $x\in B_\lambda$ we have
\[
D(ax)=D(a)x+aD(x)=\lambda(ax),
\]
so $B_\lambda$ is stable under multiplication by $a$.

On $B_\lambda$ one has
\[
(aD)(x)=aD(x)=a\lambda x=\lambda M_a(x),
\]
where $M_a$ denotes multiplication by $a$. Thus
\[
(aD)|_{B_\lambda}=\lambda\,M_a|_{B_\lambda}.
\]
Since $aD$ is semisimple, its restriction to $B_\lambda$ is semisimple. As $\lambda\neq 0$,
it follows that $M_a|_{B_\lambda}$ is semisimple, hence locally finite. In particular,
\[
\Span_k\{a^n b\mid n\ge 0\}
\]
is finite-dimensional. Since $B$ is a domain and $b\neq 0$, multiplication by $b$ is injective, so
\[
\Span_k\{a^n\mid n\ge 0\}
\]
is finite-dimensional. Hence $a$ is algebraic over $k$. Because $B$ is a finitely generated
$k$-domain and $k$ is algebraically closed, every element of $B$ algebraic over $k$ belongs to $k$.
Therefore $a\in k$.
\end{proof}

\begin{remark}
The assumption $D\neq 0$ is necessary. Indeed, if $D=0$, then $aD=0$ is semisimple for every
$a\in \ker(D)=B$, even when $a\notin k$.
\end{remark}

\begin{example}
Let $D=x\dfrac{\partial}{\partial x}-y\dfrac{\partial}{\partial y} \in \SSD(k[x,y])$.
Then $D(x)=x$, $D(y)=-y$, so $D$ is semisimple and $D(xy)=0$.
Hence $xy\in \ker(D)$.
Consider $\delta:=(xy)D$.
Although $D$ is semisimple, the derivation $\delta$ is not locally finite and therefore cannot be semisimple.
Indeed, $\delta(x)=x^2y$, $\delta(y)=-xy^2$, and the degrees of the iterates $\delta^n(x)$ and $\delta^n(y)$ grow without bound.
Thus multiplication by a nonconstant invariant may destroy semisimplicity.
\end{example}

\begin{proposition}
Let $D\in \SSD(B)$ and write
\[
B=\bigoplus_{\lambda\in \Lambda(B,D)} B_\lambda,
\qquad
\Lambda(B,D):=\{\lambda\in k\mid B_\lambda\neq 0\}.
\]
Then
\[
D(B)=\bigoplus_{\lambda\in \Lambda(B,D)\setminus\{0\}} B_\lambda.
\]
In particular,
\[
D(B)\subsetneq B.
\]
\end{proposition}

\begin{proof}
If $b\in B_\lambda$, then $D(b)=\lambda b$. Hence $D(B_\lambda)=0$ if $\lambda=0$,
and $D(B_\lambda)=B_\lambda$ if $\lambda\neq 0$.
Since $B=\bigoplus_{\lambda\in \Lambda(B,D)} B_\lambda$,
it follows that
\[
D(B)=\bigoplus_{\lambda\in \Lambda(B,D)\setminus\{0\}} B_\lambda.
\]
Moreover, $k\subseteq \ker(D)=B_0$, so $B_0\neq 0$. Therefore $D(B)$ cannot equal $B$.
\end{proof}

\section{Semi-invariants and weight decomposition}

Let $D\in \SSD(B)$. For $\lambda\in k$, define
\[
B_\lambda:=\{b\in B\mid D(b)=\lambda b\}.
\]
Elements of $B_\lambda$ will be called \emph{semi-invariants of weight $\lambda$}. In particular,
\[
B_0=\ker(D).
\]

\begin{remark}\label{rem:integer-exponents}
In the sequel, whenever expressions of the form $s^\lambda$ or $s^{-\lambda}$ appear, it is
understood that the relevant eigenvalues belong to $\Z$.
\end{remark}

\begin{proposition}
For every $\lambda,\mu\in k$, one has
\[
B_\lambda B_\mu\subseteq B_{\lambda+\mu}.
\]
In particular,
\[
B_0=\ker(D)
\]
is a $k$-subalgebra of $B$.
\end{proposition}

\begin{proof}
Let $a\in B_\lambda$ and $b\in B_\mu$. Then
\[
D(ab)=D(a)b+aD(b)=\lambda ab+\mu ab=(\lambda+\mu)ab.
\]
Hence $ab\in B_{\lambda+\mu}$. Taking $\lambda=\mu=0$, we get that $B_0$ is closed under multiplication.
It is clearly closed under addition and contains $k$.
\end{proof}

\begin{proposition}
The ring $B$ decomposes as a direct sum of eigenspaces:
\[
B=\bigoplus_{\lambda\in \Lambda(B,D)} B_\lambda,
\qquad
\Lambda(B,D):=\{\lambda\in k\mid B_\lambda\neq 0\}.
\]
\end{proposition}

\begin{proof}
This is precisely the definition of semisimplicity of $D$ as a $k$-linear endomorphism of $B$.
\end{proof}

\begin{remark}
Since $B$ is a domain, if $0\neq a\in B_\lambda$ and $0\neq b\in B_\mu$, then $ab\neq 0$, so
\[
\lambda,\mu\in \Lambda(B,D)\quad\Longrightarrow\quad \lambda+\mu\in \Lambda(B,D).
\]
\end{remark}

\begin{lemma}\label{lem:kernel-weighted-element}
Let $D\in \SSD(B)$ and let $s\in K$ satisfy
\[
D(s)=s.
\]
If $b\in K$ is a semi-invariant of weight $\lambda$, that is,
\[
D(b)=\lambda b,
\]
with $\lambda\in \Z$, then
\[
bs^{-\lambda}\in \ker(D|_K).
\]
\end{lemma}

\begin{proof}
Since $D(s)=s$, we have
\[
D(s^{-1})=-s^{-1},
\]
hence
\[
D(s^{-\lambda})=-\lambda s^{-\lambda}
\]
for every $\lambda\in \Z$ (see Remark~\ref{rem:integer-exponents}). Therefore
\[
D(bs^{-\lambda})
=
D(b)s^{-\lambda}+bD(s^{-\lambda})
=
\lambda bs^{-\lambda}-\lambda bs^{-\lambda}=0.
\]
Thus $bs^{-\lambda}\in \ker(D|_K)$.
\end{proof}

\begin{proposition}
Assume that $K=\Frac(B)$ is generated over $k$ by semi-invariants
\[
b_1,\dots,b_n\in K,
\qquad
D(b_i)=\lambda_i b_i
\]
with $\lambda_i\in \Z$. If $s\in K$ satisfies $D(s)=s$, then
\[
K=k(u_1,\dots,u_n,s),
\qquad
u_i:=b_i s^{-\lambda_i}.
\]
Moreover, each $u_i$ belongs to $\ker(D|_K)$.
\end{proposition}

\begin{proof}
By the previous lemma, each $u_i$ belongs to $\ker(D|_K)$. Since
\[
b_i=u_i s^{\lambda_i},
\]
every generator $b_i$ of $K$ belongs to $k(u_1,\dots,u_n,s)$. Hence
\[
K=k(b_1,\dots,b_n)\subseteq k(u_1,\dots,u_n,s)\subseteq K.
\]
Therefore equality holds.
\end{proof}

\begin{proposition}\label{prop:local-slice-faithful}
Let $B$ be a finitely generated $k$-domain, and let $D\in \SSD^{*}(B)$
be the semisimple derivation associated with a faithful $\mathbb G_m$-action on $\Spec(B)$.
Then there exists a nonzero semi-invariant $f\in B$
such that, in the localization
$B_f=B[f^{-1}]$,
there exists an element $s\in B_f$ satisfying $D(s)=s$.
\end{proposition}

\begin{proof}
Since $D\in \SSD^{*}(B)$, the ring $B$ decomposes as a $\mathbb Z$-graded algebra
\[
B=\bigoplus_{m\in \mathbb Z} B_m,
\]
where $D(b)=mb$ for every $b\in B_m$.

Because the corresponding $\mathbb G_m$-action is faithful, the subgroup of $\Z$
generated by the integers
\[
\{m\in \Z \mid B_m\neq 0\}
\]
must be all of $\Z$. Indeed, if these weights generated a proper subgroup $d\Z$ with $d>1$,
then every homogeneous element would have weight divisible by $d$, so the action would factor through
\[
\GM \to \GM,\qquad t\mapsto t^d.
\]
Consequently the finite subgroup $\mu_d\subseteq \GM$ would act trivially on $B$, hence trivially on $\Spec(B)$, contradicting faithfulness.

Choose homogeneous elements $a_1,\dots,a_r\in B$
with $D(a_i)=\lambda_i a_i$,
$\lambda_i\in \Z$,
such that $\gcd(\lambda_1,\dots,\lambda_r)=1$.
Then there exist integers $m_1,\dots,m_r$ satisfying
\[
m_1\lambda_1+\cdots+m_r\lambda_r=1.
\]

Set
\[
f:=\prod_{m_i<0} a_i^{-m_i}\in B
\]
and
\[
s:=a_1^{m_1}\cdots a_r^{m_r}\in B_f.
\]
Since $D(a_i)=\lambda_i a_i$, we have
\[
D(s)=\left(m_1\lambda_1+\cdots+m_r\lambda_r\right)s=s.
\]
Thus $s\in B_f$ satisfies
\[
D(s)=s.
\]
Therefore, after localizing at the semi-invariant $f$, the derivation admits an element of weight one.
\end{proof}

\begin{remark}
Proposition~\ref{prop:local-slice-faithful} shows that, after localizing at a suitable semi-invariant, one may assume the existence of an element $s$ satisfying
\[
D(s)=s.
\]
This will be the starting point for the description of kernels in the next sections.
\end{remark}

\section{Rational slices and the kernel over the field of fractions}

Let $K=\Frac(B)$, and let $D$ also denote the unique extension of a derivation
$D\in \Der_k(B)$ to $K$.

\begin{remark}\label{rem:CiLi-rational}
Let $D\in \SSD^{*}(B)$. By \cite[Definition 2.5 and Theorem 2.14]{CiLi22},
the extension of $D$ to $K$ is rational semisimple. In particular, by
\cite[Corollary 3.3]{CiLi22}, there exists an element
\[
s\in K
\]
such that
\[
D(s)=s.
\]
Thus every semisimple derivation with integer eigenvalues admits a rational slice after passing to the field of fractions.
\end{remark}

In this section we work in the situation where a rational slice has been fixed and the
field of fractions is generated by semi-invariants.

\begin{theorem}\label{thm:kernel-fraction-field}
Assume that $D\in \SSD^{*}(B)$, and let $s\in K$ be such that $D(s)=s$.
Suppose that $K$ is generated over $k$ by semi-invariants
\[
b_1,\dots,b_n\in K,
\qquad
D(b_i)=\lambda_i b_i,
\qquad
\lambda_i\in \Z.
\]
Then
\[
\ker(D)=k\left(b_1s^{-\lambda_1},\dots,b_ns^{-\lambda_n}\right),
\qquad
K=\ker(D)(s).
\]
\end{theorem}

\begin{proof}
Set
\[
u_i:=b_i s^{-\lambda_i}\qquad (i=1,\dots,n).
\]
By Lemma~\ref{lem:kernel-weighted-element}, each $u_i$ belongs to $\ker(D)$. Hence
\[
L:=k(u_1,\dots,u_n)\subseteq \ker(D).
\]

On the other hand, for each $i$ we have
\[
b_i=u_i s^{\lambda_i},
\]
so
\[
K=k(b_1,\dots,b_n)\subseteq L(s)\subseteq K.
\]
Therefore
\[
K=L(s).
\]

It remains to prove that
\[
\ker(D)=L.
\]

We claim that $s$ is transcendental over $L$. Indeed, assume that $s$ is algebraic over $L$.
Choose a nonzero polynomial
\[
P(T)=a_0+a_1T+\cdots+a_mT^m\in L[T]
\]
of minimal degree $m$ such that
\[
P(s)=0.
\]
Since $D$ vanishes on $L$ and satisfies $D(s)=s$, applying $D$ to $P(s)=0$ gives
\[
0=D(P(s))=\sum_{i=1}^m i\,a_i s^i.
\]
Because $s\neq 0$ in the field $K$, we can factor out $s$:
\[
0=\sum_{i=1}^m i\,a_i s^{i-1} = a_1 + 2a_2 s + \cdots + m a_m s^{m-1}.
\]
This is a polynomial relation over $L$ of degree $m-1$ satisfied by $s$, and its leading
coefficient $m a_m$ is nonzero since $\operatorname{char}(k)=0$ and $a_m\neq 0$.
This contradicts the minimality of $m$. Hence $s$ is transcendental over $L$.

Therefore $K=L(s)$ is a rational function field in one variable over $L$, and $D$ acts as the Euler derivation with respect to the parameter $s$, namely
\[
D=s\frac{d}{ds}.
\]
Let $f\in \ker(D)$. Since $K=L(s)$, we may write
\[
f=\frac{P(s)}{Q(s)}
\]
with $P,Q\in L[s]$ and $Q\neq 0$.

Thus
\[
0=D(f)=s\frac{df}{ds}.
\]
Since $K$ is a field and $s\neq 0$, this implies
\[
\frac{df}{ds}=0.
\]
In characteristic zero, the only rational functions in $L(s)$ with zero derivative are the elements of $L$. Hence $f\in L$. Therefore
\[
\ker(D)=L
=
k(u_1,\dots,u_n)
=
k\left(b_1s^{-\lambda_1},\dots,b_ns^{-\lambda_n}\right).
\]

The equality
\[
K=\ker(D)(s)
\]
now follows from
\[
K=L(s)
\qquad\text{and}\qquad
L=\ker(D).
\]
\end{proof}

\begin{proposition}\label{prop:kernel-localization}
Let $D\in \SSD^{*}(B)$ and assume that there exists an element $s\in B$
such that $D(s)=s$.
Let
\[
B=k[b_1,\dots,b_n],
\qquad
D(b_i)=\lambda_i b_i \quad (\lambda_i\in \Z),
\]
and set
\[
B_s:=B[s^{-1}].
\]
Then
\[
\ker(D|_{B_s})=
k\!\left[b_1s^{-\lambda_1},\dots,b_ns^{-\lambda_n}\right].
\]
\end{proposition}

\begin{proof}
Set
\[
u_i:=b_i s^{-\lambda_i}\in B_s.
\]
Since $D(s)=s$, one has
\[
D(s^{-1})=-s^{-1},
\]
and therefore
\[
D(u_i)=D(b_i)s^{-\lambda_i}+b_iD(s^{-\lambda_i})
=\lambda_i b_i s^{-\lambda_i}-\lambda_i b_i s^{-\lambda_i}=0.
\]
Hence
\[
k[u_1,\dots,u_n]\subseteq \ker(D|_{B_s}).
\]

Conversely, let $f\in B_s$. Since $B_s=k[b_1,\dots,b_n,s^{-1}]$,
the element $f$ can be written as a finite sum of monomials
\[
f=\sum_{\alpha,m} c_{\alpha,m}\, b^\alpha s^{-m},
\]
where $\alpha=(\alpha_1,\dots,\alpha_n)\in \Z_{\ge 0}^n$, $m\in\Z_{\ge 0}$, and
\[
b^\alpha=b_1^{\alpha_1}\cdots b_n^{\alpha_n}.
\]
Each monomial is a semi-invariant for $D$, and
\[
D(b^\alpha s^{-m})
=
\left(\sum_{i=1}^n \alpha_i\lambda_i-m\right)b^\alpha s^{-m}.
\]
Thus $f\in \ker(D|_{B_s})$ if and only if every monomial appearing in $f$ has weight zero, namely
\[
m=\sum_{i=1}^n \alpha_i\lambda_i.
\]
For such a monomial we get
\[
b^\alpha s^{-m}
=
\prod_{i=1}^n (b_i s^{-\lambda_i})^{\alpha_i}
=
u_1^{\alpha_1}\cdots u_n^{\alpha_n}.
\]
Therefore every element of $\ker(D|_{B_s})$ belongs to $k[u_1,\dots,u_n]$, and the result follows.
\end{proof}

\begin{corollary}\label{cor:kernel-intersection}
With the same notation,
\[
\ker(D)=\ker(D|_{B_s})\cap B
=
k\!\left[b_1s^{-\lambda_1},\dots,b_ns^{-\lambda_n}\right]\cap B.
\]
\end{corollary}

\begin{proof}
The inclusion $\ker(D)\subseteq \ker(D|_{B_s})\cap B$
is immediate. Conversely, if
\[
f\in \ker(D|_{B_s})\cap B,
\]
then $f\in B$ and
$D(f)=D|_{B_s}(f)=0$,
so $f\in \ker(D)$.
\end{proof}

\end{document}